\documentclass[11pt,reqno]{amsart}

\usepackage[T1]{fontenc}
\usepackage{lmodern}
\usepackage{amssymb}
\usepackage{microtype}
\usepackage[letterpaper,margin=1.1in]{geometry}
\usepackage{enumitem}
\usepackage{cite}
\usepackage[hidelinks]{hyperref}

\hypersetup{
 pdftitle={Local Presentations of Groups: Completion and Smith’s Conjecture},
 pdfsubject={Local-map completion, preservation of presentations, and a sharp uniform obstruction}
}

\newtheorem{theorem}{Theorem}[section]
\newtheorem{lemma}[theorem]{Lemma}
\newtheorem{proposition}[theorem]{Proposition}
\newtheorem{corollary}[theorem]{Corollary}
\theoremstyle{definition}
\newtheorem{definition}[theorem]{Definition}
\newtheorem{remark}[theorem]{Remark}
\numberwithin{equation}{section}

\DeclareMathOperator{\Aff}{Aff}
\DeclareMathOperator{\ord}{ord}
\DeclareMathOperator{\id}{id}

\newcommand{\Qbar}{\overline{\mathbb{Q}}}
\newcommand{\gen}[1]{\langle #1\rangle}
\newcommand{\case}[1]{\par\smallskip\noindent\emph{#1}\enspace}
\setlist[enumerate]{label=\textup{(\roman*)},leftmargin=*,itemsep=2pt}

\title[Local Presentations of Groups]{Local Presentations of Groups: Completion\\and Smith's Conjecture}
\author{Yu Pan}
\address{School of Mathematical Sciences, Eastern Institute of Technology, Ningbo, Zhejiang, 315200, China}
\email{AllisonPan97@outlook.com}
\date{}

\begin{document}

\begin{abstract}
    For a symmetric subset of a group containing the identity, its universal local presentation records the multiplication relations visible within the subset. We prove that a group is recovered from every such generating subset if and only if it is trivial or cyclic of order two or three, as conjectured by P.~A.~Smith. The main step is an exact completion theorem: local maps into a target group always extend to generating domains if and only if every target element has a square root. By embedding an arbitrary group in a group with surjective squaring, we deduce that every partial local presentation embeds in a generating one. This gives subgroup inheritance and reduces the classification to finite generation. We also construct obstructions in one fixed countable metabelian target on domains of at most five elements. This bound is sharp even when the target is allowed to vary.
\end{abstract}
\maketitle

\section{Introduction}\label{sec:introduction}

The question studied here is whether a group can be recovered from the multiplication visible in each of its generating subsets. We first give the definitions needed to state it precisely.

Let $G$ be a group. A subset $X\subseteq G$ is \emph{symmetric} if $X^{-1}=X$. Throughout, a \emph{domain} means a symmetric subset containing $1$; it need not be a subgroup. Define
\begin{equation}\label{eq:presentation}
    U(X)=\bigl\langle [x]\ (x\in X)\;\bigm|\;
    [1]=1,\quad [x][y]=[xy]\ \text{whenever }x,y,xy\in X\bigr\rangle.
\end{equation}

The symbols $[x]$ are formal generators, and only the displayed relations are imposed. Sending $[x]$ to $x$ respects these relations and gives a canonical homomorphism
\begin{equation*}
    \pi_X\colon U(X)\longrightarrow\gen{X},\qquad [x]\longmapsto x.
\end{equation*}
It is surjective because its image contains $X$. A nontrivial element of its kernel represents a relation in $G$ which does not follow from the multiplication rules visible in $X$. We write $P(G)$ if $\pi_X\colon U(X)\to G$ is an isomorphism for every domain $X$ generating $G$. Thus, $P(G)$ asks whether every relation in $G$ follows from the local table of any such $X$.

Smith proposed that the only groups simply connected relative to every generating local subgroup are the cyclic groups of orders one, two, and three \cite[p. 438]{smith1950}. Here, a local subgroup is a domain with the inherited partial multiplication. His local identity words give the same presentation as \eqref{eq:presentation}: if $z=xy\in X$, the word $xyz^{-1}$ is a local identity, since both $xy=z$ and $yz^{-1}=x^{-1}$ belong to $X$. Conversely, evaluating a local identity word through its intermediate products decomposes the corresponding relation into the relations in \eqref{eq:presentation} and reduces it to $[1]=1$. The two sets of relators therefore generate the same normal subgroup.

\begin{theorem}\label{thm:classification}
    For an arbitrary group $G$, the following are equivalent:
    \begin{enumerate}
        \item $P(G)$ holds;
        \item $G\cong1$, $C_2$, or $C_3$.
    \end{enumerate}
\end{theorem}

The quantifier over all generating domains is essential. For comparison, under the hypotheses of \cite[Proposition 2.1]{soergel2023}, simple connectedness of a Cayley flag complex is characterized by a triangular presentation for a fixed generating set. Here, we work directly with \eqref{eq:presentation}, without assuming finite generation, countability, or freeness of a simplicial action.

Section \ref{sec:completion} proves the local map completion theorem. Section \ref{sec:preservation} uses it to preserve partial presentations and show that if $P(G)$ holds, then $P(H)$ holds for every subgroup $H\le G$. This reduces the classification to finitely generated groups, which we treat in Section \ref{sec:classification}. Section \ref{sec:detection} gives a second proof using small affine obstructions and proves that the bound of five elements is sharp.

We work in ZFC, using Zorn's lemma for maximal extensions. All domain-size bounds count the identity; $1=C_1$ denotes the trivial group.

\begin{remark}
    The classification has two proofs: preservation and subgroup inheritance reduce it to finite generation, while the affine obstructions give a direct route from the completion theorem. The preservation theorem is stronger than retaining a single obstruction, since it embeds the whole partial universal group.
    The use of Zorn's lemma does not show that choice is necessary for the classification, and gives no general algorithm for completion. Natural further questions concern effective constructions, bounds on completed domains, smaller detecting targets, and the relation to Smith's group complexes. These questions are distinct from the sharp bound on initial domains proved here.
\end{remark}

\vskip0.5cm
\noindent \textit{Acknowledgements}.
The author is grateful to Professor Tianyu Yuan for his support, guidance, and helpful suggestions.
The author acknowledges the use of ChatGPT for exploration of proof strategies and assistance with developing and checking arguments.  
The author subsequently revised the exposition and takes full responsibility for the mathematical content and the final text.

\section{An exact completion theorem}\label{sec:completion}

We first consider maps which respect the multiplication visible in a domain. Such maps allow us to test whether the local presentation contains all the relations of the ambient group.

\begin{definition}\label{def:local}
    For a domain $X\subseteq G$ and a group $K$, a map $f\colon X\to K$ is
    \emph{local} if $f(1)=1$ and
    \begin{equation*}
        f(xy)=f(x)f(y)\qquad\text{whenever }x,y,xy\in X.
    \end{equation*}
\end{definition}

A local map preserves inverses, since
$f(x)f(x^{-1})=f(1)=1$. By the universal property of a presentation, local maps $X\to K$ correspond bijectively to homomorphisms $U(X)\to K$: the assignment $[x]\mapsto f(x)$ respects the defining relations, and a homomorphism is determined by its values on these generators. Consequently, if $P(G)$ holds, every local map on a generating domain extends uniquely to a homomorphism on $G$, by composing the induced map with $\pi_X^{-1}$.

Our next aim is to extend a local map to a domain which generates $G$. The extension need only remain local; it need not be defined on all of $G$. We first consider adjoining one inverse pair. 
If $t^2$ already lies in the domain, the image of $t$ must be a square root of the prescribed image of $t^2$, which explains the hypothesis below.

\begin{lemma}\label{lem:enlargement}
    Suppose every element of $K$ has a square root. If $f\colon X\to K$ is local and $t\in G\setminus\gen{X}$, then $f$ extends to a local map on $X\cup\{t,t^{-1}\}$.
\end{lemma}

\begin{proof}
    Put $H=\gen{X}$ and $Y=X\cup\{t,t^{-1}\}$. If $t^2\in X$, choose $b\in K$ with $b^2=f(t^2)$; otherwise set $b=1$. Define
    \begin{equation*}
        F|_X=f,\qquad F(t^\varepsilon)=b^\varepsilon
        \quad(\varepsilon=\pm1).
    \end{equation*}
    Neither $t$ nor $t^{-1}$ lies in $X$. If they coincide, then $t^2=1\in X$ and $b^2=1$, so the definition is still well-defined. We check all visible products in $Y$.

    \case{Case 1: two old factors.}
    Their product lies in $H$, so cannot be $t^{\pm1}$. If it belongs to $Y$, it therefore belongs to $X$, and locality follows from $f$.

    \case{Case 2: one old and one new factor.}
    For $x\in X$, neither $xt^\varepsilon$ nor $t^\varepsilon x$ can belong to $X$, since either equality would put $t$ in $H$. If the product equals a new element $t^\eta$, then
    \begin{equation*}
        xt^\varepsilon=t^\eta\ \text{or}\ t^\varepsilon x=t^\eta
        \quad\Longrightarrow\quad x=t^{\eta-\varepsilon},
        \qquad\varepsilon,\eta\in\{1,-1\}.
    \end{equation*}
    Thus $x\in\{1,t^2,t^{-2}\}$. If $x=1$, then $f(x)=1$. If $x=t^2$ or $t^{-2}$, symmetry gives $t^2\in X$, and the choice of $b$ together with inverse preservation gives $f(x)=b^2$ or $b^{-2}$, respectively. In each case, $f(x)=b^{\eta-\varepsilon}$. Hence
    \begin{equation*}
        f(x)b^\varepsilon=b^\eta=b^\varepsilon f(x).
    \end{equation*}
    Only powers of $b$ are involved, so no commutativity assumption on $K$ is needed.

    \case{Case 3: two new factors.}
    Opposite signs give inverse cancellation. Equal signs give $t^{2\varepsilon}$: if this lies in $X$, the square equation or its inverse applies. If $t^{2\varepsilon}=t^\eta$ is new, then $\eta=\varepsilon$ would force $t=1$, while $\eta=-\varepsilon$ forces $t^3=1$. In the latter case $t^2=t^{-1}\notin X$, so $b=1$ and the relation is respected. These cases prove that $F$ is local.
\end{proof}

The preceding lemma gives the extension step. We now show that the existence of square roots is both sufficient and necessary for completion to a generating domain.

\begin{theorem}\label{thm:completion}
    For a group $K$, the following are equivalent:
    \begin{enumerate}
        \item For every group $G$, every domain $X_0\subseteq G$, and every local map $f_0\colon X_0\to K$, there are a generating domain $X\supseteq X_0$ and a local map $f\colon X\to K$ extending $f_0$.
        \item The assertion in \textup{(i)} holds whenever $X_0$ is finite.
        \item The function $K\to K$, $k\mapsto k^2$, is surjective.
    \end{enumerate}
\end{theorem}

\begin{proof}
    The implication (i)$\Rightarrow$(ii) is immediate. For (ii)$\Rightarrow$(iii), let $a\in K$ and take $C_8=\gen{g}$ with
    \begin{equation*}
        X_0=\{1,g^2,g^{-2}\},\qquad
        f_0(g^2)=a,\quad f_0(g^{-2})=a^{-1},\quad f_0(1)=1.
    \end{equation*}
    This map is local: the two nonidentity squares equal $g^4\notin X_0$, so only identity and inverse relations are visible. By (ii), $f_0$ extends to a local map $f$ on a generating domain $X$. Such an $X$ contains an odd power $t=g^r$, since the even powers form a proper subgroup. Then $t^2$ is $g^2$ or $g^{-2}$, giving $f(t)^2=a$ or $f(t)^2=a^{-1}$. In the second case $(f(t)^{-1})^2=a$. Thus every $a$ has a square root.

    Now assume (iii), and fix $G,X_0,f_0$. Consider all pairs $(Y,h)$, where $Y$ is a domain containing $X_0$ and $h\colon Y\to K$ is a local extension of $f_0$. 
    Order these pairs by extension: $(Y,h)\le (Z,k)$ if $Y\subseteq Z$ and $k|_Y=h$. These pairs form a set, since their graphs lie in $G\times K$, and the set is nonempty because it contains $(X_0,f_0)$.

    For a nonempty chain, take the union of its domains and maps. Comparability ensures that the maps agree wherever both are defined, so the union is a function. Its domain is symmetric and contains $X_0$. 
    Any triple $u,v,uv$ in the union lies in one chain member: choose a member for each element and take the largest of the three. Locality in that member proves locality of the union. This gives an upper bound; the empty chain has upper bound $(X_0,f_0)$.

    Zorn's lemma therefore gives a maximal local extension $f\colon X\to K$. If $\gen{X}\ne G$, choose $t\notin\gen{X}$. Lemma \ref{lem:enlargement} extends $f$ to the strictly larger domain $X\cup\{t,t^{-1}\}$, a contradiction. Hence $X$ generates $G$, proving (i).
\end{proof}

\begin{remark}\label{rem:squaring}
    For abelian groups, surjective squaring is called \emph{$2$-divisibility}. For finite groups it is equivalent to odd order: if $N=|K|$ is odd, then $(a^{(N+1)/2})^2=a^{N+1}=a$ by Lagrange's theorem. 
    Conversely, a surjection on a finite set is injective. If $|K|$ were even, Cauchy's theorem would give an involution $v\ne1$ with $v^2=1^2$, a contradiction. For nonabelian $K$, squaring need not be a homomorphism; only its surjectivity is used.
\end{remark}

\section{Preservation of partial presentations}\label{sec:preservation}

Enlarging a domain introduces new multiplication relations, which may make a nontrivial element of its original presentation trivial. We now show that every domain $X_0$ has a generating enlargement $X$ for which this does not happen: the natural map $U(X_0)\to U(X)$ is injective. To apply the completion theorem, we first construct a target containing $U(X_0)$ in which every element has a square root.

\begin{lemma}\label{lem:root-hull}
    Every group $K$ embeds in a group $R(K)$ in which every element has a square root.
\end{lemma}

\begin{proof}
    Form $W(K)=(K\times K)\rtimes C_2$, where the nonidentity element of $C_2$ swaps the coordinates. For $\varepsilon,\eta\in\{0,1\}$, the multiplication is
    \begin{equation*}
        (a,b;\varepsilon)(c,d;\eta)=
        \begin{cases}
            (ac,bd;\varepsilon+\eta),&\varepsilon=0,\\
            (ad,bc;\varepsilon+\eta),&\varepsilon=1,
        \end{cases}
    \end{equation*}
    with the last coordinate taken modulo two. The diagonal map $\delta_K(k)=(k,k;0)$ is an injective homomorphism, and
    \begin{equation}\label{eq:wreath-square}
        (k,1;1)^2=(k,k;0)=\delta_K(k).
    \end{equation}

    Thus every element of the embedded copy of $K$ has a square root in $W(K)$. To give square roots to the new elements as well, we repeat the construction. Set $K_0=K$, $K_{n+1}=W(K_n)$, with embeddings $\delta_{K_n}$, and take their direct limit $R(K)$. Concretely, an element of $R(K)$ is represented by a pair $(k,n)$ with $k\in K_n$; two pairs represent the same element if their images agree at a common later stage. Products and inverses are computed after passing to a common stage. These operations are well-defined because the structure maps are homomorphisms, and the group laws hold because they hold at every stage.

    Each canonical map $i_n\colon K_n\to R(K)$ is injective: equality of two images is equality at a later stage, where the composite structure map is injective. Every element of the limit is some $i_n(k)$, and \eqref{eq:wreath-square} gives
    \begin{equation*}
        \bigl(i_{n+1}(k,1;1)\bigr)^2
        =i_{n+1}(\delta_{K_n}(k))=i_n(k).
    \end{equation*}
    Thus $i_0$ embeds $K$ and every element of $R(K)$ has a square root.
\end{proof}

\begin{theorem}\label{thm:preservation}
    For every group $G$ and every domain $X_0\subseteq G$, there is a generating domain $X\supseteq X_0$ such that the natural homomorphism
    \begin{equation*}
        j\colon U(X_0)\longrightarrow U(X),\qquad [x]\longmapsto[x],
    \end{equation*}
    is injective.
\end{theorem}

\begin{proof}
    Put $K=U(X_0)$ and embed it by $\iota\colon K\hookrightarrow R(K)$ using Lemma \ref{lem:root-hull}. The map $x\mapsto\iota([x])$ on $X_0$ is local, so Theorem \ref{thm:completion} extends it to a local map $f\colon X\to R(K)$ on a generating domain of $G$. It induces $\Phi\colon U(X)\to R(K)$ with $\Phi([x])=f(x)$.

    The map $j$ is well-defined because every multiplication triple in $X_0$ is also one in $X$. For $x\in X_0$,
    \begin{equation*}
        (\Phi\circ j)([x])=f(x)=\iota([x]).
    \end{equation*}
    These symbols generate $K$, so $\Phi\circ j=\iota$. If $j(w)=1$, then $\iota(w)=1$, and injectivity of $\iota$ gives $w=1$. Hence $j$ is injective.
\end{proof}

The preservation theorem allows us to transfer the property $P$ from a group to each of its subgroups.

\begin{corollary}\label{cor:subgroups}
    If $P(G)$ holds and $H\le G$, then $P(H)$ holds.
\end{corollary}

\begin{proof}
    Let $Y$ be any generating domain of $H$. Applying Theorem \ref{thm:preservation} inside $G$ gives a generating domain $X\supseteq Y$ of $G$ and an injection $j\colon U(Y)\to U(X)$. If $\iota\colon H\hookrightarrow G$ is inclusion, evaluation on the generators gives
    \begin{equation*}
        \pi_X\circ j=\iota\circ\pi_Y.
    \end{equation*}
    By $P(G)$, $\pi_X$ is an isomorphism, so $\pi_X\circ j$ is injective. The identity above then implies that $\pi_Y$ is injective. It is also surjective because $Y$ generates $H$. Thus $P(H)$ holds.
\end{proof}

\begin{remark}\label{rem:preservation}
    Theorem \ref{thm:preservation} preserves every nontrivial element of $U(X_0)$ simultaneously. It asserts the existence of a suitable enlargement, not that every enlargement has this property or that the resulting generating domain is finite.
\end{remark}

\section{Proof of the classification}\label{sec:classification}

By subgroup inheritance, it will be enough to classify the finitely generated groups satisfying $P$: any group with at least four elements contains a finitely generated subgroup with at least four elements. We therefore begin with a finite generating set. A finite generating set is \emph{irredundant} if none of its elements belongs to the subgroup generated by the others. Deleting redundant elements from a finite generating set produces such a set without changing the group generated.

\begin{lemma}\label{lem:irredundant}
    Let $S=\{s_1,\dots,s_d\}$ be an irredundant finite generating set of $G$, and put $X=\{1\}\cup S\cup S^{-1}$. Then
    \begin{equation*}
        U(X)\cong B_1*\cdots*B_d,\qquad
        B_i=\begin{cases}
            C_2,&\ord(s_i)=2,\\
            C_3,&\ord(s_i)=3,\\
            \mathbb{Z},&\text{otherwise}.
        \end{cases}
    \end{equation*}
    The designated generator of $B_i$ corresponds to $[s_i]$. The empty free product is understood to be trivial.
\end{lemma}

\begin{proof}
    Irredundance excludes $s_i=1$ and $s_i=s_j^{-1}$ for distinct indices. Put $A_i=[s_i]$. Eliminate $[1]$ and the inverse symbols using $[1]=1$ and $[s_i^{-1}]=A_i^{-1}$; when $s_i=s_i^{-1}$, this imposes $A_i^2=1$. Each remaining defining relator has length at most three in the letters $A_i^{\pm1}$.

    If a relator uses different indices, some index occurs exactly once, since the relator has length at most three. Its evaluation can be written $u s_i^\varepsilon v=1$, where $\varepsilon=\pm1$ and $u,v$ use only the other generators. This gives $s_i^\varepsilon=u^{-1}v^{-1}$, contradicting irredundance. Thus no mixed relator occurs.

    A relator using one index reduces to $A_i^e$ with $|e|\le3$. The case $e=0$ is cancellation, while $|e|=1$ would give $s_i=1$. The cases $|e|=2,3$ are possible only at orders two and three, respectively. Conversely, $s_i^2=1$ is visible at order two, and $s_i^2=s_i^{-1}$ is visible at order three, giving $A_i^2=1$ or $A_i^3=1$. There are no other relators. The presentation therefore splits into the asserted free product.
\end{proof}

Recall that a nonempty reduced word in a free product is nontrivial: its letters are nonidentity elements of the factors, and consecutive letters belong to distinct factors. We use this normal form to find explicit elements in the kernels of local evaluations.

\begin{proposition}\label{prop:finite-generation}
    A finitely generated group satisfying $P$ is isomorphic to $1$, $C_2$, or $C_3$.
\end{proposition}

\begin{proof}
    Choose $S$ and $X$ as in Lemma \ref{lem:irredundant}. By $P(G)$, evaluation identifies $G$ with $B_1*\cdots*B_d$. We exclude the other possibilities by choosing a new generating domain whose local relations omit a relation that holds in $G$.

    \case{Case 1: an infinite cyclic factor is present.}
    Write $G=\gen{s}*J$, where $s$ generates that factor and $J$ is the free product of the others. Let $D$ consist of $1$ and the designated generators and their inverses in the factors in $J$. Replace $s^{\pm1}$ by $s^{\pm2},s^{\pm3}$, giving
    \begin{equation*}
        E=\{1,s^2,s^{-2},s^3,s^{-3}\},\qquad Y=D\cup E.
    \end{equation*}
    This domain still generates $G$, because $s=s^3(s^2)^{-1}$.

    Among exponents $0,\pm2,\pm3$, the only sums remaining in this set come from adding zero or opposite exponents; all other sums are $\pm1,\pm4,\pm5,\pm6$. Since $s$ has infinite order, the only visible relations in $E$ are therefore identity and inverse relations. A product of nonidentity elements from $D$ and $E$ has reduced length two in $J*\gen{s}$, so cannot lie in $Y$. Moreover, $Y\cap J=D$ and $Y\cap\langle s\rangle=E$, so the remaining visible products are precisely those in $D$ and those in $E$. Thus the local presentation splits as
    \begin{equation*}
        U(Y)\cong J*\gen{u}*\gen{v},\qquad u=[s^2],\quad v=[s^3],
    \end{equation*}
    where the last two factors are infinite cyclic; the relations in $D$ present $J$ by Lemma \ref{lem:irredundant}. In $G$, the elements $s^2$ and $s^3$ satisfy $(s^2)^3=(s^3)^2$, but this relation does not follow from the local table. Indeed, the word $u^3v^{-2}$ is reduced and nontrivial, but
    \begin{equation*}
        \pi_Y(u^3v^{-2})=(s^2)^3(s^3)^{-2}=1.
    \end{equation*}
    This contradicts $P(G)$.

    \case{Case 2: at least two factors of orders two or three remain.}
    Choose generators $a,b$ of distinct factors, of orders $m,n\in\{2,3\}$. Write $G=J*\gen{b}$, with $a\in J$, and let $D$ be the domain of designated generators and inverses for $J$, including $1$. Set
    \begin{equation*}
        t=ab,\qquad Y=D\cup\{t,t^{-1}\}.
    \end{equation*}
    Again $Y$ generates, since $b=a^{-1}t$. The words $(ab)^r$ and $(b^{-1}a^{-1})^r$ are reduced of length $2r$ for $r>0$. Hence $t$ has infinite order, and $t^{\pm1},t^{\pm2}\notin J$.

    Two elements of $D$ multiply into $J$, so cannot give $t^{\pm1}$. A mixed product $xt^\varepsilon$ or $t^\varepsilon x$, with $x\in D$, cannot lie in $D$, since that would put $t$ in $J$. If it equals $t^\eta$, then $x=t^{\eta-\varepsilon}$; as $t^{\pm2}\notin J$, only $x=1$ is possible. Finally, new factors with opposite signs cancel, while their equal-sign products $t^{\pm2}$ are neither old elements nor $t^{\pm1}$. Thus
    \begin{equation*}
        U(Y)\cong J*\gen{T},\qquad T=[t],
    \end{equation*}
    where $\gen{T}$ is infinite cyclic. For $A=[a]$, the word $(A^{-1}T)^n$ is reduced of length $2n$ and nontrivial, whereas
    \begin{equation*}
        \pi_Y\bigl((A^{-1}T)^n\bigr)=(a^{-1}ab)^n=b^n=1.
    \end{equation*}
    This is another contradiction. Therefore there are no infinite cyclic factors and at most one finite factor. The possibilities are precisely $1$, $C_2$, and $C_3$.
\end{proof}

\begin{proof}[Proof of Theorem \ref{thm:classification}]
    Suppose $P(G)$ holds. If $G$ contains four distinct elements, the subgroup $H$ they generate is finitely generated and has at least four elements. By Corollary \ref{cor:subgroups}, $P(H)$ holds, contrary to Proposition \ref{prop:finite-generation}. Hence $|G|\le3$, so $G$ is trivial or cyclic of order two or three.

    Conversely, a generating domain in any of these groups is the whole group: in $C_3$, containing one nonidentity element also forces its distinct inverse, and the other cases are immediate. The full multiplication table presents $G$, since $g\mapsto[g]$ is then a homomorphism inverse to evaluation. Thus $P(G)$ holds.
\end{proof}

\section{A fixed target and a sharp five-element bound}\label{sec:detection}

We now give a second proof of the classification by constructing local maps on small domains which cannot extend to homomorphisms on $G$. The completion theorem will extend these maps to generating domains and thereby show that $P(G)$ fails. For every $G$ with $|G|\ge4$, we will use the same target group and show that five elements suffice for the initial domain.

\begin{definition}\label{def:obstruction}
    A \emph{local obstruction with target $K$} is a local map $f_0\colon X_0\to K$, where $X_0\subseteq G$ is a domain, which is not the restriction of any homomorphism $G\to K$.
\end{definition}

\begin{lemma}\label{lem:obstruction-completion}
    Suppose every element of $K$ has a square root. Every local obstruction $f_0\colon X_0\to K$ extends locally to a generating domain $X$ of $G$. For every such extension, $\pi_X\colon U(X)\to G$ has nontrivial kernel.
\end{lemma}

\begin{proof}
    Existence follows from Theorem \ref{thm:completion}. Any local extension $f$ induces $\Phi\colon U(X)\to K$. If $\pi_X$ had trivial kernel, it would be an isomorphism, and $\Phi\circ\pi_X^{-1}$ would be a homomorphism $G\to K$ extending $f_0$, a contradiction.
\end{proof}

We next choose a target to which the preceding lemma applies. Let $\Qbar$ be the field of algebraic numbers, and consider the affine group
\begin{equation*}
    F=\Aff_1(\Qbar)=\{h_{a,b}\colon z\mapsto az+b:
    a\in\Qbar^{\times},\ b\in\Qbar\}.
\end{equation*}
Its operation is composition, with $AB=A\circ B$. The formulas
\begin{equation*}
    h_{a,b}\circ h_{c,d}=h_{ac,ad+b},\qquad
    h_{a,b}^{-1}=h_{a^{-1},-a^{-1}b}
\end{equation*}
show that it is a group and that the multiplier
$\mu\colon F\to\Qbar^{\times}$, $\mu(h_{a,b})=a$, is a surjective homomorphism.

\begin{lemma}\label{lem:affine-target}
    The group $F$ is countable and metabelian, and every element has a square root.
\end{lemma}

\begin{proof}
    The algebraic numbers form a countable set: they are roots of the countably many nonzero integer polynomials, each with finitely many roots. Hence $F$ is countable. The kernel of $\mu$ is the abelian translation group, and the quotient is the abelian group $\Qbar^{\times}$. Thus the commutator subgroup lies in $\ker\mu$ and is abelian. This is precisely the assertion that $F$ is metabelian.

    For $h_{a,b}\in F$, choose $u\in\Qbar$ with $u^2=a$ and $u\ne-1$. This is possible because $a\ne0$ has two distinct square roots; when $a=1$, take $u=1$. Then $q(z)=uz+b/(u+1)$ belongs to $F$, and
    \begin{equation*}
        (q\circ q)(z)=u^2z+\frac{ub+b}{u+1}=az+b.
    \end{equation*}
    Thus every affine map has a square root in $F$.
\end{proof}

\begin{samepage}
    \begin{theorem}\label{thm:detection}
        For a group $G$, the following are equivalent:
        \begin{enumerate}
            \item $|G|\ge4$;
            \item $P(G)$ fails;
            \item There exists a local obstruction $f_0\colon X_0\to F$ with $|X_0|\le5$.
        \end{enumerate}
        Every obstruction in \textup{(iii)} has a local generating completion, and every such completion has noninjective evaluation. The bound five is sharp, even if the detecting target is allowed to depend on $G$.
    \end{theorem}
\end{samepage}

\begin{proof}
    We first prove (i)$\Rightarrow$(iii). Assume $|G|\ge4$. We construct the obstruction by considering four cases. Write $\zeta_r$ for a primitive $r$-th root of unity in $\Qbar$. All assignments below preserve the identity and inverses.

    First note the relations on $D_c=\{1,c,c^{-1}\}$ for $c\ne1$. At order two the only additional relation is $[c]^2=1$; at order three it is $[c]^3=1$. At any other order, $c^{\pm2}\notin D_c$: an equality $c^2=1,c$, or $c^{-1}$ would force order two, one, or three. Thus only identity and inverse relations are visible in that case.

    \case{Case 1: $G$ is nonabelian.}
    Choose a nonidentity commutator $c=ghg^{-1}h^{-1}$. On $X_0=D_c$, send $c$ to the dilation $z\mapsto\lambda z$, where
    \begin{equation*}
        \lambda=\begin{cases}
            \zeta_3,&\ord(c)=3,\\
            -1,&\text{otherwise}.
        \end{cases}
    \end{equation*}
    This respects the order-three relation, the involution relation when present, and all inverse relations. Hence the map is local on at most three elements. A global extension $\varphi\colon G\to F$ would give a homomorphism $\mu\circ\varphi$ to an abelian group, which must kill every commutator. But its value on $c$ is $\lambda\ne1$, a contradiction.

    \case{Case 2: an element has finite order $n\ge4$.}
    For such an element $s$, take $X_0=D_s$ and send $s$ to the dilation with multiplier $\zeta_{n+1}$. The description of $D_s$ proves locality. A homomorphism extending it would send $s^n=1$ to the identity, forcing $\zeta_{n+1}^{\,n}=1$. This is impossible because $\zeta_{n+1}$ has order $n+1$. Here $|X_0|=3$.

    \case{Case 3: an element has infinite order.}
    For such an $s$, take
    \begin{equation*}
        X_0=\{1,s^2,s^{-2},s^3,s^{-3}\},\qquad
        f_0(s^2)=\id,\quad f_0(s^3)(z)=\zeta_3z.
    \end{equation*}
    Among the exponents $0,\pm2,\pm3$, a sum remains in this set only when one summand is zero or the summands are opposite. Since $s$ has infinite order, the domain therefore has only identity and inverse relations, so $f_0$ is local. If it extended to $\varphi\colon G\to F$, the multiplier $\lambda=\mu(\varphi(s))$ would satisfy
    \begin{equation*}
        \lambda^2=1,\qquad \lambda^3=\zeta_3.
    \end{equation*}
    The first equation gives $\lambda=\pm1$, whose cube cannot be $\zeta_3$. This gives an obstruction on five elements.

    \case{Case 4: the remaining abelian groups.}
    If none of the preceding cases applies, $G$ is abelian and every nonidentity element has order two or three. Both orders cannot occur: if $u,v$ have these orders, then $\gen{u}\cap\gen{v}=\{1\}$. From $(uv)^r=1$ we obtain $u^r=v^{-r}$ in this intersection, so $2\mid r$ and $3\mid r$. Since $(uv)^6=1$, their product has order six, already excluded.

    Thus $G$ has exponent $p\in\{2,3\}$ and is a vector space over $\mathbb{F}_p$, with scalar multiplication given by integer powers. Choose $a\ne1$ and $b\notin\gen{a}$, which is possible because $|\gen{a}|=p<|G|$. These elements are independent. Indeed, if $a^i b^j=1$ and $j\ne0$ in $\mathbb{F}_p$, choose an integer $k$ with $jk\equiv1\pmod{p}$. Raising $b^j=a^{-i}$ to the power $k$ gives $b=a^{-ik}\in\langle a\rangle$, a contradiction. Hence $j=0$ and then $i=0$ in $\mathbb{F}_p$.

    We will map $a$ and $b$ to elements of order $p$ which do not commute. Set $\zeta=\zeta_p$ and define
    \begin{equation*}
        A(z)=\zeta z,\qquad B(z)=\zeta z+(1-\zeta).
    \end{equation*}
    Both have order $p$, since $B=TAT^{-1}$ for the translation $T(z)=z+1$. They do not commute:
    \begin{equation*}
        AB(z)=\zeta^2z+\zeta(1-\zeta),\qquad
        BA(z)=\zeta^2z+(1-\zeta),
    \end{equation*}
    and the difference of the constant terms is $-(\zeta-1)^2\ne0$.

    For $p=2$, use $X_0=\{1,a,b\}$ and assign $a\mapsto A$, $b\mapsto B$. The domain is symmetric, $ab\notin X_0$ by independence, and its only nonidentity relations $a^2=b^2=1$ are respected.

    For $p=3$, use $X_0=\{1,a,a^{-1},b,b^{-1}\}$ with the same assignments and inverse values. The restrictions to $\gen{a}$ and $\gen{b}$ are homomorphisms. No mixed nonidentity product lies in $X_0$: such a product has both coordinates nonzero in the span of $a,b$, while an element of $X_0$ has at least one coordinate zero. Hence both maps are local, on three and five elements respectively. Neither extends to a homomorphism on $G$, since homomorphic images of the commuting elements $a,b$ would commute.

    These cases prove (i)$\Rightarrow$(iii). By Lemmas \ref{lem:affine-target} and \ref{lem:obstruction-completion}, (iii) implies (ii), together with the assertion about every generating completion. If $|G|\le3$, then $G=1,C_2$, or $C_3$ up to isomorphism, and $P(G)$ holds by the direct verification in Theorem \ref{thm:classification}. Thus (ii)$\Rightarrow$(i). In fact every domain in these groups is either $\{1\}$ or $G$, so every local map extends, with any target. This proves the equivalence without using the assertion from Section \ref{sec:classification} that $P(G)$ forces $|G|\le3$.

    It remains to show that the bound five cannot be reduced, even by changing the target. Let $G=C_3\times C_3$ and let $K$ be any group. A domain $X_0$ of size at most four has size one or three, since its nonidentity elements come in distinct inverse pairs. Hence $X_0=\{1\}$ or $X_0=\{1,a,a^{-1}\}=\gen{a}$. It is therefore a subgroup, and every local map $f_0\colon X_0\to K$ is a homomorphism. To extend $f_0$, it suffices to find a homomorphism $r\colon G\to X_0$ which fixes every element of $X_0$, called a retraction. Use the trivial map when $X_0=\{1\}$; otherwise choose $b\notin\gen{a}$ and set
    \begin{equation*}
        r(a^i b^j)=a^i\qquad(i,j\in\mathbb{F}_3).
    \end{equation*}
    The pair $a,b$ is a basis of this two-dimensional space, so the expression is unique and $r$ is a homomorphism fixing $X_0$. Thus $f_0\circ r$ extends $f_0$. No target can detect this group on four or fewer elements, while Case 4 supplies a five-element obstruction. This proves the optimality of the bound.
\end{proof}

\begin{remark}\label{rem:initial-domain}
    The five-element bound concerns the initial domain $X_0$. Its generating completion may be larger or infinite. In particular, the theorem applies to groups which have no finite generating set.
\end{remark}

\bibliographystyle{amsplain}
\begingroup
    \raggedright
    \bibliography{smith}

@inproceedings{smith1950,
  author    = {Smith, P. A.},
  title     = {Some topological notions connected with a set of generators},
  booktitle = {Proceedings of the International Congress of Mathematicians,
               Cambridge, Massachusetts, 1950},
  volume    = {II},
  publisher = {American Mathematical Society},
  year      = {1952},
  pages     = {436--441},
  url       = {https://www.mathunion.org/fileadmin/ICM/Proceedings/ICM1950.2/ICM1950.2.ocr.pdf},
  note      = {\href{https://www.mathunion.org/fileadmin/ICM/Proceedings/ICM1950.2/ICM1950.2.ocr.pdf}{Official proceedings}}
}

@article{soergel2023,
  author  = {Soergel, M.},
  title   = {Systolic complexes and group presentations},
  journal = {Groups Geom. Dyn.},
  volume  = {17},
  year    = {2023},
  pages   = {899--921},
  doi     = {10.4171/GGD/717},
  note    = {\href{https://doi.org/10.4171/GGD/717}{doi:10.4171/GGD/717}}
}
\endgroup
\end{document}